\documentclass[12pt]{amsart}
\usepackage{amssymb}
\usepackage{amsfonts}
\usepackage{amssymb,latexsym}
\usepackage{enumerate}
\usepackage{url}
\usepackage{mathrsfs}
\usepackage{bbm}
\makeatletter
\@namedef{subjclassname@2010}{%
  \textup{2010} Mathematics Subject Classification}
\makeatother

  [2002/01/19 v2.2g %
    AMS font definitions%
  ]
\DeclareFontFamily{U}{euf}{}
\DeclareFontShape{U}{euf}{m}{n}{%
  <5><6><7><8><9>gen*eufm%
  <10><10.95><12><14.4><17.28><20.74><24.88>eufm10%
  }{}
\DeclareFontShape{U}{euf}{b}{n}{%
  <5><6><7><8><9>gen*eufb%
  <10><10.95><12><14.4><17.28><20.74><24.88>eufb10%
  }{}

  [2002/01/19 v2.2g %
    AMS font definitions%
  ]
\DeclareFontFamily{U}{msb}{}
\DeclareFontShape{U}{msb}{m}{n}{%
  <5><6><7><8><9>gen*msbm%
  <10><10.95><12><14.4><17.28><20.74><24.88>msbm10%
  }{}

  [2002/01/19 v2.2g %
    AMS font definitions%
  ]
\DeclareFontFamily{U}{msa}{}
\DeclareFontShape{U}{msa}{m}{n}{%
  <5><6><7><8><9>gen*msam%
  <10><10.95><12><14.4><17.28><20.74><24.88>msam10%
  }{}

\newtheorem{theorem}{Theorem}[section]
\newtheorem{lemma}[theorem]{Lemma}
\newtheorem{proposition}[theorem]{Proposition}
\newtheorem{corollary}[theorem]{Corollary}

\theoremstyle{definition}

\newtheorem{example}[theorem]{Example}

\newtheorem{remark}[theorem]{Remark}

\numberwithin{equation}{section} 

\def\C{\mathcal C}
\def\Cl{{\rm Cl}}

\begin{document}

\title[Euler's integral and multiple cosine functions]
{Euler's integral, multiple cosine function and zeta values}

\author{Su Hu}
\address{Department of Mathematics, South China University of Technology, Guangzhou, Guangdong 510640, China}
\email{mahusu@scut.edu.cn}

\author{Min-Soo Kim}
\address{Department of Mathematics Education, Kyungnam University, Changwon, Gyeongnam 51767, Republic of Korea}
\email{mskim@kyungnam.ac.kr}
\thanks{This work was supported by the National Research Foundation of Korea (NRF) grant funded by the Korea government (MSIT) (No. NRF-2022R1F1A1065551). }


\begin{abstract}
In 1769, Euler proved the following result
$$ \int_0^{\frac\pi2}\log(\sin \theta) d\theta=-\frac\pi2 \log2.$$
In this paper, as a generalization, we evaluate the definite integrals
$$\int_0^x \theta^{r-2}\log\left(\cos\frac\theta2\right)d\theta $$
for $r=2,3,4,\ldots.$ 
We show that it can be expressed by the special values of Kurokawa and Koyama's multiple cosine functions  
$\mathcal{C}_r(x)$ or by the special values of alternating zeta
and Dirichlet lambda functions.

In particular, we get the following explicit expression of the zeta value 
$$\zeta(3)=\frac{4\pi^2}{21}\log\left(\frac{e^{\frac{4G}{\pi}}\mathcal{C}_3\left(\frac14\right)^{16}}{\sqrt2}\right),$$
where $G$ is Catalan's constant and $\mathcal{C}_3\left(\frac14\right)$ is the special value of Kurokawa and Koyama's multiple cosine function
$\mathcal{C}_3(x)$ at $\frac14$. 
Furthermore, we prove several  series representations for the logarithm of 
multiple cosine functions  $\log\mathcal{C}_r\left(\frac x{2}\right)$ by zeta functions, $L$-functions or polylogarithms.
One of them leads to another expression of
 $\zeta(3)$:
$$\zeta(3)=\frac{72\pi^2}{11}\log\left(\frac{3^{\frac1{72}}\C_3\left(\frac16\right)}{\C_2\left(\frac16\right)^{\frac13}}\right).$$
\end{abstract}

\maketitle

\section{Introduction}
\subsection{Zeta functions}
The main purpose of this paper is to relate the Euler type integrals and the multiple cosine functions with the special values of zeta functions.
So in this section, to our purpose, firstly we introduce various types of zeta functions.

For Re$(s)>1,$  the Riemann zeta function is defined by
\begin{equation}~\label{Ri-zeta}
\zeta(s)=\sum_{n=1}^{\infty}\frac{1}{n^{s}}.
\end{equation}
This function can be analytically continued to a meromorphic function  in the
complex plane except for a simple pole, with residue 1, at the point $s=1$.
The special number $\zeta(3)=1.20205\cdots$ is called Ap\'ery constant.
It is named after  Ap\'ery, who proved in 1979 that $\zeta(3)$ is irrational (see \cite{Ape}).
 
For Re$(s)>1$  and  $a\neq0,-1,-2,\ldots,$ in 1882, Hurwitz \cite{Hurwitz} defined the partial zeta function
\begin{equation}~\label{Hurwitz}
\zeta(s,a)=\sum_{n=0}^{\infty}\frac{1}{(n+a)^{s}},
\end{equation}
which generalized (\ref{Ri-zeta}). 
As (\ref{Ri-zeta}), this function can also be analytically continued to a meromorphic function  in the
complex plane except for a simple pole at $s=1$ with residue 1.

The alternating Hurwitz zeta function is defined by
\begin{equation}\label{E-zeta-def}
\zeta_E(s,a)=\sum_{n=0}^\infty\frac{(-1)^n}{(n+a)^{s}},
\end{equation}
where Re$(s)>0$ and  $a\neq0,-1,-2,\ldots$ (see \cite{Ayo}, \cite{Cvj} and \cite{SK2019}).
It can be analytically continued to the complex plane without any pole.
Sometimes we may use the notation $J(s,a)$ instead of $\zeta_E(s,a)$
(see, e.g.,  Williams and Zhang \cite[p. 36, (1.1)]{WZ}).
There exists the following relationship between $\zeta_E(s,a)$ and $\zeta(s,a)$ (see \cite[p. 37, (2.3)]{WZ}):
\begin{equation}\label{E-zeta-HZ}
\zeta_E(s,a)=2^{-s}\left(\zeta\left(s,\frac a2\right)-\zeta\left(s,\frac{a+1}2\right)\right).
\end{equation}
Recently, the Fourier expansion and several integral representations, special values and power series expansions,
convexity properties  of $\zeta_{E}(s,a)$ have been investigated (see \cite{Cvj, HKK, SK2019}), and  it has been found that  $\zeta_{E}(s,a)$ can be used to represent a partial zeta function of cyclotomic fields in one version of Stark's conjectures in algebraic number theory (see \cite[p. 4249, (6.13)]{HK-G}).

In particular setting $a=1,$ the function $\zeta_E(s,a)$ reduces to the alternating zeta function $\zeta_{E}(s)$
(also known as Dirichlet's eta or Euler’s eta function),
\begin{equation}\label{A-zeta-1}
\zeta_E(s)=\sum_{n=1}^\infty\frac{(-1)^{n+1}}{n^{s}}=\eta(s).
\end{equation}
Obviously, \begin{equation}\label{A-zeta}
\zeta_E(s)=(1-2^{1-s})\zeta(s).
\end{equation}
And from the Taylor expansion of $\log(1+x)$, we have
\begin{equation}\label{e-zeta(1)}
\zeta_E(1)=\sum_{n=1}^\infty\frac{(-1)^{n+1}}{n}=\log2
\end{equation}
(see \cite{Ayo}, \cite{Cvj} and \cite{SK2019}).
According to Weil's history~\cite[p.~273--276]{Weil}, the function $\zeta_E(s)$ has been used by Euler to ``prove"
\begin{equation}~\label{fe}
\frac{\zeta_{E}(1-s)}{\zeta_{E}(s)}=-\frac{\Gamma(s)(2^{s}-1)\textrm{cos}(\pi s/2)}{(2^{s-1}-1)\pi^{s}}
\end{equation}
which leads to the functional equation of $\zeta(s)$.
It is  also a particular case of Witten's zeta functions in mathematical physics \cite[p. 248, (3.14)]{Min}, and  it has been studied and evaluated at certain positive integers by Sitaramachandra Rao \cite{SiR} in terms of the Riemann zeta values. See also \cite[p. 31, \S 7]{FS} and \cite[p. 2, (2)]{Mi}.

The Dirichlet lambda function $\lambda(s)$ is defined by
\begin{equation}\label{lam}
\begin{aligned}
\lambda(s)&=\sum_{n=0}^\infty\frac1{(2n+1)^s}\\
&=\frac1{2^s}\zeta\left(s,\frac12\right)=(1-2^{-s})\zeta(s)
\end{aligned}
\end{equation}
for Re$(s)>1$ (see \cite[p. 954, (1.9)]{SK2019}). This function was studied by Euler under the notation $N(s)$ 
(see \cite[p. 70]{Var}).
 Euler also considered its alternating form
\begin{equation}\label{beta-def}
\beta(s)=\sum_{n=0}^\infty\frac{(-1)^{n}}{(2n+1)^s}=\frac1{2^s}\zeta_E\left(s,\frac12\right)
\end{equation} 
for Re$(s)>0$,
which he denoted by $L(s)$ (see \cite[p.~70]{Var}). 
Furthermore, the constant $\beta(2)= G$ is usually named as Catalan's constant 
(see \cite{Ku06}, \cite{RZ}, \cite{SC} and \cite{WG}).
Both functions admit an analytic continuation, $\lambda(s)$ to all $s\neq1$ and $\beta(s)$ to all $s.$
They have been studied in detail by us in \cite{SK2019}, in particular, we have obtained a number of infinite families of 
linear recurrence relations for $\lambda(s)$ at positive even integer arguments $\lambda(2m)$,
convolution identities for special values of $\lambda(s)$ at even arguments and special values
of $\beta(s)$ at odd arguments.

The Dirichlet $L$-function associated to a Dirichlet character $\chi$  is given by
\begin{equation}\label{di-def}
\begin{aligned}
L(s,\chi)&=\sum_{n=1}^\infty\frac{\chi(n)}{n^s}\\
&=\prod_{p}\frac1{1-\chi(p)p^{-s}},
\end{aligned}
\end{equation}
which is convergent for Re$(s)>1$ and the Euler product is taken over all prime numbers $p.$ 
It was introduced by Dirichlet in 1837 to prove the theorem on primes in arithmetic progressions (see \cite[Chapter 7]{Apostol}).
For the trivial Dirichlet character $\mathbbm{1}$ we have $L(s,\mathbbm{1})=\zeta(s).$
For the principal character $\mathbbm{1}_m$ of modulus $m$ induced by $\mathbbm{1}$ we have
\cite[p. 255]{IR}
\begin{equation}\label{prin-L}
\zeta(s)=L(s,\mathbbm{1}_m)\prod_{p\mid m}\frac1{1-p^{-s}}.
\end{equation}
We may also express $L(s,\chi)$ by using the Hurwitz zeta functions. Let $f$ be a positive integer and let $\chi$ be any character modulo $f.$ 
The Dirichlet $L$-function $L(s,\chi)$ is expressed in terms of the Hurwitz zeta function $\zeta(s, a)$ by means of the following formula
\begin{equation}\label{L-H}
L(s,\chi)=f^{-s}\sum_{a=1}^{f-1}\chi(a)\zeta\left(s,\frac af\right)
\end{equation}
for Re$(s)>1,$ and it can be analytic continued to the whole $s$-plane from the above expression.

\subsection{Multiple trigonometric functions and the related integrals} 
Around 1742, Euler successfully calculated the zeta value
\begin{equation}\label{Mengoli}
\zeta(2)=\sum_{n=1}^{\infty}\frac{1}{n^2}=1+\frac{1}{4}+\frac{1}{9}+\cdots
\end{equation}
by considering the integration 
\begin{equation}\label{Euler2}
\frac{1}{2}(\arcsin x)^{2}=\int_{0}^{x}\frac{\arcsin t}{\sqrt{1-t^{2}}}dt.
\end{equation}
 In concrete, by taking $x=1$ in the left-hand side we get $\frac{\pi^{2}}{8}$, and
by expanding $\arcsin t$ as a power series and integrating term-by-term on the right-hand side,  
 we get the sum
 $$\lambda(2)=1+\frac{1}{3^{2}}+\frac{1}{5^{2}}+\cdots,$$
 where $\lambda(s)$ is the Dirichlet lambda function (see (\ref{lam})).
 Then by comparing the results on the both sides we arrive at the summation
 \begin{equation}\label{Euler3}
\lambda(2)=1+\frac{1}{3^{2}}+\frac{1}{5^{2}}+\cdots=\frac{\pi^{2}}{8}.
\end{equation}
Finally, the identity
\begin{equation}\label{Euler4}
1+\frac{1}{3^{2}}+\frac{1}{5^{2}}+\cdots=\zeta(2)-\frac{1}{2^{2}}\zeta(2)=\frac{3}{4}\zeta(2)
\end{equation}
leads to
 \begin{equation}
\zeta(2)=1+\frac{1}{4}+\frac{1}{9}+\cdots=\frac{\pi^2}{6}.
\end{equation}
And more generally, for $n=1,2,3,\ldots,$ Euler obtained
\begin{equation} \label{zeta-even}
\zeta(2n)=1+\frac{1}{2^{2n}}+\frac{1}{3^{2n}}+\cdots=\frac{(-1)^{n-1}B_{2n}2^{2n}}{2(2n)!}\pi^{2n},
\end{equation}
where the $B_{2n}$ are the Bernoulli numbers defined by the generating function
\begin{equation}
\frac{t}{e^{t}-1}=\sum_{n=0}^{\infty}B_{n}\frac{t^n}{n!}.
\end{equation}
(See \cite[p. 266, Theorem 12.17]{Apostol}). 
But the explicit formulas for $\zeta(3)$ and $\zeta(2n+1)$ are still unknown.
For the long-standing history, we refer to a recent book by Nahin \cite{Nahin}.

To extend (\ref{Euler3}) from 2 to 3, Euler got the following formula
\begin{equation}\label{Euler5}
\lambda(3)=1+\frac{1}{3^{3}}+\frac{1}{5^3}+\cdots=\frac{\pi^2}{\log 2}+2\int_0^{\frac{\pi}{2}}\theta\log(\sin \theta) d\theta
\end{equation}
(see \cite[p. 63]{Var}), and in 1769 Euler \cite{Eu} proved the following result 
\begin{equation}\label{Eu-int}
I=\int_0^{\frac\pi2}\log(\sin \theta) d\theta=-\frac\pi2 \log2,
\end{equation}
which is equal  to
\begin{equation}\label{Eu-int-c}
I=\int_0^{\frac\pi2}\log(\cos \theta) d\theta
\end{equation}
(see \cite[p. 152]{KW04}). 

Generalizing  the above integrals (\ref{Euler5}) and (\ref{Eu-int}), for $0\leq x<\pi$ and $r=2,3,4,\ldots,$ Koyama and Kurokawa \cite{KK05} evaluated the definite integrals
\begin{equation}\label{Eu-int-gen}
\int_0^{x}\theta^{r-2}\log(\sin \theta) d\theta
\end{equation}
and showed that (\ref{Eu-int-gen}) is expressed by the multiple sine functions:
\begin{equation}\label{K-K}
\int_0^x \theta^{r-2}\log(\sin\theta)d\theta=\frac{x^{r-1}}{r-1}\log\left(\sin x\right)
-\frac{\pi^{r-1}}{r-1}\log \mathcal S_r\left(\frac x \pi\right)
\end{equation}
(see \cite[Theorem 1]{KK05}).

It is well-known that the definition of sine functions starts from the following infinite product representation
\begin{equation}
\sin x=x \prod_{n=1}^{\infty} \left(1-\frac{x^2}{n^2\pi^{2}}\right)
\end{equation} 
(see \cite[p. 44, 1.431(1)]{GR} and \cite[p. 28, (1.4.9)]{Nahin}).
Denote by 
$$\mathcal S_1(x)=2\sin(\pi x)=2\pi x \prod_{n=1}^\infty\left(1-\frac{x^2}{n^2}\right).$$
In 1886, as a generalization, H\"older \cite{Ho} defined the double sine function $\mathcal S_2(x)$ from the infinite product
\begin{equation}
\mathcal S_2(x)=e^x\prod_{n=1}^\infty\left\{\left(\frac{1-\frac xn}{1+\frac xn}\right)^ne^{2x}\right\}.
\end{equation}
Then in 1990s, generalizing $\mathcal S_2(x),$ Kurokawa \cite{Ku1,Ku2,Ku3} further defined
the multiple sine function $\mathcal S_r(x)$ of order $r=2,3,4,\ldots$ by the Weierstrass product
\begin{equation}
\mathcal S_r(x)=\exp\left(\frac{x^{r-1}}{r-1}\right)\prod_{n=1}^\infty\left\{P_r\left(\frac xn\right)P_r\left(-\frac xn\right)^{(-1)^{r-1}}\right\}^{n^{r-1}},
\end{equation}
where 
\begin{equation}\label{Pr}
P_r(x)=(1-x)\exp\left(x+\frac{x^2}{2}+\cdots+\frac{x^r}{r}\right).
\end{equation}

The cosine function has the following infinite product representation
\begin{equation}\label{cose} 
\cos x=\prod_{n=1, n\text{:odd}}^\infty \left(1-\frac{x^2}{(\frac{n\pi}2)^2} \right)
\end{equation}
(see \cite[p. 45, 1.431(3)]{GR}). Denote by
\begin{equation}\label{cos1-ex}
\C_1(x)=2\cos(\pi x) \\
=2\prod_{n=1, n\text{:odd}}^\infty \left(1-\frac{x^2}{(\frac n2)^2} \right).
\end{equation}
Then in 2003, Kurokawa and Koyama \cite{KK} defined the multiple cosine function $\C_r(x)$ from the Weierstrass product
\begin{equation}\label{mcos-def}
\begin{aligned}
\C_r(x)&=\prod_{n=-\infty, n\text{:odd}}^\infty P_r\left(\frac{x}{\frac n2}\right)^{(\frac n2)^{r-1}} \\
&=\prod_{n=1, n\text{:odd}}^\infty \left\{P_r\left(\frac{x}{\frac n2}\right)P_r\left(-\frac{x}{\frac n2}\right)^{(-1)^{r-1}}\right\}^{(\frac n2)^{r-1}}
\end{aligned}
\end{equation}
for $r=2,3,4,\ldots$ (see also \cite{KW03}, \cite{KW} and \cite{KW04}).

Letting $r=2,3$ and 4 in (\ref{mcos-def}), we get
\begin{equation}\label{mcos-ex}
\begin{aligned}
\C_2(x)&=\prod_{n=1, n\text{:odd}}^\infty \left\{\left(\frac{1-\frac x{(\frac n2)}}{1+\frac x{(\frac n2)}}\right)^{\frac n2}e^{2x}\right\}, 
\\
\C_3(x)&=\prod_{n=1, n\text{:odd}}^\infty \left\{\left(1-\frac{x^2}{(\frac n2)^2}\right)^{(\frac n2)^2}e^{x^2}\right\},
\\
\C_4(x)&=\prod_{n=1, n\text{:odd}}^\infty \left\{\left(\frac{1-\frac x{(\frac n2)}}{1+\frac x{(\frac n2)}}\right)^{(\frac n2)^3}e^{\frac{n^2}2 x+\frac23 x^3}\right\}
\end{aligned}
\end{equation}
(see \cite{KK,KW03,KW04}).
Then the duplication formulas are expressed as
\begin{equation}\label{sin-cos}
\C_r(x)^{2^{r-1}}=\frac{\mathcal S_r(2x)}{\mathcal S_r(x)^{2^{r-1}}}
\end{equation}
for $r\geq1$ (see \cite[p. 848]{KK}, \cite[p. 125]{KW03}, \cite[p. 477]{KW} and \cite[p. 142]{KW04}).
The proof of (\ref{sin-cos}) can be found in \cite[p. 125, \S3]{KW03} and Corollary \ref{dup} below.

\subsection{Our results} 
In this paper, inspiring by Koyama and Kurokawa's work  \cite{KK05},  we evaluate the definite integrals
\begin{equation}\label{mainp}
 \int_0^x \theta^{r-2}\log\left(\cos\frac\theta2\right)d\theta \end{equation}
for $r=2,3,4,\ldots.$ 
We show that (\ref{mainp}) can be expressed by the special values of Kurokawa and Koyama's multiple cosine functions  
$\mathcal{C}_r(x)$ (see Theorem \ref{thm1}) or by the special values of alternating zeta
and Dirichlet lambda functions (see Theorems \ref{thm1-gen} and \ref{thm1-cor}).

In particular, we get the following explicit expression of the zeta value 
\begin{equation}\label{cos-zeta}
\zeta(3)=\frac{4\pi^2}{21}\log\left(\frac{e^{\frac{4G}{\pi}}\C_3\left(\frac14\right)^{16}}{\sqrt2}\right),
\end{equation}
where $G$ is Catalan's constant and $\C_3\left(\frac14\right)$ is the special value of Kurokawa and Koyama's multiple cosine function
$\C_3(x)$ at $\frac14$ (see Corollary \ref{zeta3}). 
As pointed by Allouche in an email to us, the above identity is equivalent to the following formula by Kurokawa and Wakayama (see \cite[p. 123]{KW03}):
$$\C_3\left(\frac14\right)=2^{\frac{1}{32}}\exp\left(\frac{21\zeta(3)}{64\pi^{2}}-\frac{L(2,\chi_{-4})}{4\pi}\right),$$
where  $L(2,\chi_{-4})$ equals to the Catalan constant $G$. Recently,  following (\ref{cos-zeta}), 
Allouche \cite{All} found a link between the Kurokawa multiple trigonometric functions and two functions introduced respectively by Borwein--Dykshoorn \cite{Borwein} and by Adamchik \cite{Adamchik 2005}.

Furthermore, we prove several series representations
 of $\log\C_r\left(\frac x{2}\right)$ by $\lambda(2n)$ for $n=1,2,3,\ldots$ or by $\zeta_E(r)$ for $r=2,3,4,\ldots$  and
 the special values of polylogarithms (see Theorems \ref{integ-2} and  \ref{integ-poly}).
 From Theorem \ref{integ-poly}, we express the special values of $\C_2\left(\frac16\right)$
 and $\C_3\left(\frac16\right)$ by $L(2,\chi_3), L(2,\chi_6),$ the special values of Dirichlet's $L$-functions, 
 and the zeta value $\zeta(3)$ (see Corollary \ref{c(1/6)}). This leads to another expression  of
 $\zeta(3)$:
$$\zeta(3)=\frac{72\pi^2}{11}\log\left(\frac{3^{\frac1{72}}\C_3\left(\frac16\right)}{\C_2\left(\frac16\right)^{\frac13}}\right)$$
 (see Remark \ref{1.6}).
 
\section{Main results}\label{main results}
In this section, we state our main results. Their  proofs will be given in Section \ref{proofs}.
First, we represent (\ref{mainp}) by the special values of multiple cosine functions.

\begin{theorem}\label{thm1}
For $0\leq x<\pi$ and $r=2,3,4,\ldots,$ we have
$$
\int_0^x \theta^{r-2}\log\left(\cos\frac\theta2\right)d\theta=\frac{x^{r-1}}{r-1}\log\left(\cos\frac x2\right)
-\frac{(2\pi)^{r-1}}{r-1}\log\C_r\left(\frac x{2\pi}\right).
$$
\end{theorem}
Letting $x=\frac{\pi}{2}$ in Theorem \ref{thm1}, we have

\begin{corollary}\label{thm1-co}
For $r=2,3,4,\ldots,$
$$
\int_0^{\frac\pi2} \theta^{r-2}\log\left(\cos\frac\theta2\right)d\theta=-\frac{\pi^{r-1}}{r-1}\left(\frac1{2^r}\log2+
2^{r-1}\log\C_r\left(\frac14\right)\right).
$$
\end{corollary}

Setting  $r=2,3$ and 4 in Corollary \ref{thm1-co} respectively, and by (\ref{mcos-ex}) with $x=\frac14$ we get the following examples:
$$
\begin{aligned}
\int_0^{\frac\pi2}\log\left(\cos\frac\theta2\right)d\theta
&=\frac\pi2\log\frac1{\sqrt2}-2\pi \log\C_2\left(\frac14\right) \\
&=\frac\pi2\log\frac1{\sqrt2}-\log\left(\prod_{n=1,n{\rm :odd}}^\infty
\left\{\left(\frac{2n-1}{2n+1}\right)^{\frac n2}e^{\frac12}\right\}\right)^{2\pi},
\end{aligned}
$$
$$
\begin{aligned}
\int_0^{\frac\pi2}\theta\log\left(\cos\frac\theta2\right)d\theta&=\frac{\pi^2}8\log\frac1{\sqrt2}- 2\pi^2\log\C_3\left(\frac14\right) \\
&=\frac{\pi^2}8\log\frac1{\sqrt2}-\log\left(\prod_{n=1,n\text{:odd}}^\infty\left(1-\frac{1}{4n^2}\right)^{\frac{n^2}4}e^{\frac1{16}}\right)^{2\pi^2},
\end{aligned}
$$
$$
\begin{aligned}
\int_0^{\frac\pi2}\theta^2\log\left(\cos\frac\theta2\right)d\theta&=\frac{\pi^3}{24}\log\frac1{\sqrt2}- \frac{8\pi^3}3 \log\C_4\left(\frac14\right) \\
&=\frac{\pi^3}{24}\log\frac1{\sqrt2}-\log\left(\prod_{n=1,n\text{:odd}}^\infty\left(\frac{2n-1}{2n+1} \right)^{\frac{n^3}8}e^{\frac{n^2}{8}+\frac1{96}}\right)^{\frac{8\pi^3}3}.
\end{aligned}
$$

In the following, we shall employ the usual convention that an empty sum is taken to be zero. For example, if $n=0,$ then we understand that $\sum_{k=1}^{n}=0.$
We represent  (\ref{mainp}) with $x=\frac\pi2$ by the special values of alternating zeta, lambda and beta functions.

Now we state the following result.

\begin{theorem}\label{thm1-gen}
For $r=2,3,4,\ldots,$
$$
\begin{aligned}
\int_0^{\frac\pi2} \theta^{r-2}\log\left(\cos\frac\theta2\right)d\theta
&=-\frac{\log2}{r-1}\left(\frac\pi2\right)^{r-1}+(r-2)!\sin\left(\frac{r\pi}{2}\right)\zeta_E(r) \\
&\quad+\sum_{k=0}^{\left\lfloor \frac{r-2}{2}\right\rfloor }(-1)^k(2k)!\binom{r-2}{2k}
\left(\frac\pi2\right)^{r-2k-2} \\
&\quad\times\beta(2k+2) \\
&\quad+\sum_{k=1}^{\left\lceil \frac{r-2}2\right\rceil}\frac{(-1)^{k-1}(2k-1)!}{2^{2k+1}}\binom{r-2}{2k-1}
\left(\frac\pi2\right)^{r-2k-1} \\
&\quad\times\zeta_E(2k+1),
\end{aligned}
$$
where $\lfloor x\rfloor =\max\{m\in \mathbb {Z} \mid m\leq x\}$ and $\lceil x\rceil =\min\{m\in \mathbb {Z} \mid m\geq x\}.$
\end{theorem}

Combining  Corollary \ref{thm1-co} and Theorem \ref{thm1-gen}, we arrive at the following theorem.

\begin{theorem}\label{thm1-cor}
For $r=2,3,4,\ldots,$
$$
\begin{aligned}
\log\C_r\left(\frac14\right)
&=\frac{\log2}{2^{2r-1}}-\frac{(r-1)!}{(2\pi)^{r-1}}\sin\left(\frac{r\pi}{2}\right)\zeta_E(r) \\
&\quad-\frac{r-1}{2^{2(r-1)}}\sum_{k=0}^{\left\lfloor \frac{r-2}{2}\right\rfloor }(-1)^k(2k)!\binom{r-2}{2k}\left(\frac2\pi\right)^{2k+1}
 \\
&\quad\times\beta(2k+2) \\
&\quad-\frac{r-1}{2^{2r-1}}\sum_{k=1}^{\left\lceil \frac{r-2}2\right\rceil}\frac{(-1)^{k-1}(2k-1)!}{\pi^{2k}}\binom{r-2}{2k-1}
\\
&\quad\times\zeta_E(2k+1).
\end{aligned}
$$
\end{theorem}

The Catalan constant 
$$
G=\sum_{n=0}^\infty\frac{(-1)^n}{(2n+1)^2}=0.915965594177219015\cdots
$$
is one of famous mysterious constants appearing in many places in mathematics and physics. It can be represented 
by the special values of Hurwitz zeta functions
\begin{equation}\label{beta-zeta-2}
G=\beta(2)=\frac14\zeta_E\left(2,\frac12\right)=\frac1{16}\left(\zeta\left(2,\frac14\right)-\zeta\left(2,\frac34\right)\right)
\end{equation}
(see \cite[p. 667, (1.1)]{Ku06} and \cite[p. 29, (16)]{SC}).

\begin{example}\label{E-int-ex}
From Theorem \ref{thm1-gen} with $r=2,3,4,5$ and (\ref{beta-zeta-2}),
we have the following examples:
$$
\begin{aligned}
\int_0^{\frac\pi2} \theta^0\log\left(\cos\frac\theta2\right)d\theta&=-\frac{\pi\log2}2 +G, \\
\int_0^{\frac\pi2} \theta^1\log\left(\cos\frac\theta2\right)d\theta&=-\frac{\pi^2\log2}{8}+\frac{\pi G}{2}-\frac{7\zeta_E(3)}{8}, \\
\int_0^{\frac\pi2} \theta^2\log\left(\cos\frac\theta2\right)d\theta
&=-\frac{\pi^3\log2}{24}+\frac{\pi^2 G}{4}+\frac{\pi\zeta_E(3)}{12}-2\beta(4), \\
\int_0^{\frac\pi2} \theta^3\log\left(\cos\frac\theta2\right)d\theta
&=-\frac{\pi^4\log2}{64}+\frac{\pi^3 G}{8}+\frac{3\pi^2\zeta_E(3)}{32}-3\pi\beta(4) \\
&\quad+\frac{93\zeta_E(5)}{16}.
\end{aligned}
$$
\end{example}

Setting  $r=2,3,4$ and 5 in Theorem \ref{thm1-cor}, by (\ref{beta-zeta-2}) we get the following corollary.

\begin{corollary}\label{rem-ex}
$$
\begin{aligned}
& \log\C_2\left(\frac14\right)=\frac{\log2}{8}-\frac{G}{2\pi}, \\
&\log\C_3\left(\frac14\right)=\frac{\log2}{32}-\frac{G}{4\pi}+\frac{7\zeta_E(3)}{16\pi^2}, \\
&\log\C_4\left(\frac14\right)=\frac{\log2}{128}-\frac{3G}{32\pi}-\frac{3\zeta_E(3)}{64\pi^2}+\frac{3\beta(4)}{4\pi^3}, \\
&\log\C_5\left(\frac14\right)=\frac{\log2}{512}-\frac{G}{32\pi}-\frac{3\zeta_E(3)}{128\pi^2}
+\frac{3\beta(4)}{4\pi^3}-\frac{93\zeta_E(5)}{64\pi^4}.
\end{aligned}
$$
\end{corollary}

From Corollary \ref{rem-ex} for $r=3$ and (\ref{A-zeta}) we have the following expression for $\zeta(3).$

\begin{corollary}\label{zeta3}
$$
\zeta(3)=\frac{4\pi^2}{21}\log\left(\frac{e^{\frac{4G}{\pi}}\C_3\left(\frac14\right)^{16}}{\sqrt2}\right).
$$
\end{corollary}

We also get the following infinite series representation of $\log\C_r\left(\frac x{2\pi}\right)$ by $\lambda(2n)$
for $n=1,2,3,\ldots.$ 

\begin{theorem}\label{integ-2}
For $0\leq x<\pi$ and $r=2,3,4,\ldots,$ we have
$$
\begin{aligned}
\log\C_r\left(\frac x{2\pi}\right)&=\left(\frac x{2\pi}\right)^{r-1}
\left(\log\left(\cos\frac x2\right)+(r-1)\sum_{n=1}^\infty\frac{\lambda(2n)}{n(2n+r-1)}\left(\frac x{\pi}\right)^{2n}\right).
\end{aligned}
$$
\end{theorem}

Setting $x=\frac\pi2$ in Theorem \ref{integ-2}, then we have
\begin{corollary}\label{integ-2-cor}
For $r=2,3,4,\ldots,$
$$
\begin{aligned}
\log\C_r\left(\frac14\right)&=\left(\frac14\right)^{r-1}
\left(-\frac12\log2+(r-1)\sum_{n=1}^\infty\frac{\lambda(2n)}{n(2n+r-1)2^{2n}}\right).
\end{aligned}
$$
\end{corollary}

This corollary gives the following examples.

\begin{example}
Setting $r=2,$ we get
$$
\log\C_2\left(\frac14\right)=\frac14
\left(-\frac12\log2+\sum_{n=1}^\infty\frac{\lambda(2n)}{n(2n+1)2^{2n}}\right).
$$
Moreover, by Corollary \ref{rem-ex} we have
$$\log\C_2\left(\frac14\right)=\frac18\log2-\frac{G}{2\pi}.$$
Therefore \cite[p. 244, (694)]{SC}
$$\sum_{n=1}^\infty\frac{\lambda(2n)}{n(2n+1)2^{2n}}
=\sum_{n=1}^\infty\frac{(2^{2n}-1)\zeta(2n)}{n(2n+1)2^{4n}}=\log2-\frac{2G}\pi,$$
where we have used (\ref{lam}).
Similarly, for $r=3,4$ and 5, we find that
$$
\begin{aligned}
\sum_{n=1}^\infty\frac{\lambda(2n)}{n(2n+2)2^{2n}}&=\frac12\log2-\frac{2G}\pi+\frac{7\zeta_E(3)}{2\pi^2}, \\
\sum_{n=1}^\infty\frac{\lambda(2n)}{n(2n+3)2^{2n}}&=\frac13\log2-\frac{2G}\pi-\frac{\zeta_E(3)}{\pi^2}
+\frac{16\beta(4)}{\pi^3}, \\
\sum_{n=1}^\infty\frac{\lambda(2n)}{n(2n+4)2^{2n}}&=\frac1{4}\log2-\frac{2G}{\pi}-\frac{3\zeta_E(3)}{2\pi^2}
-\frac{93\zeta_E(5)}{\pi^4} \\
&\quad+\frac{48\beta(4)}{\pi^3}.
\end{aligned}
$$
\end{example}

Finally, we state series representations of $\log\C_r\left(\frac x{2}\right)$ by $\zeta_E(r)$ and the special values of polylogarithms $\textrm{Li}_k(x)$.

Recall that the polylogarithm function $\textrm{Li}_k(x)$ is defined by
\begin{equation}\label{poly-def}
\textrm{Li}_k(x)=\sum_{n=1}^\infty\frac{x^n}{n^k},
\end{equation}
where $|x|<1$ and $k=1,2,3,\ldots$ (see \cite{Ku3} and \cite{KK}).

\begin{theorem}\label{integ-poly}
Let $r=2,3,4,\ldots.$ Then
\begin{enumerate}
\item[\rm(1)]  We have
$$
\begin{aligned}
\log\C_r\left(\frac x{2}\right)&=-\frac{(r-1)!}{(2\pi i)^{r-1}}\sum_{k=0}^{r-1}\frac{(\pi i)^k}{k!}{\rm Li}_{r-k}(-e^{-\pi ix})x^k
+\frac{\pi i}{r}\left(\frac x2\right)^{r} \\
&\quad-\frac{(r-1)!}{(2\pi i)^{r-1}}\zeta_E(r)
\end{aligned}
$$
for ${\rm Im}(x)<0.$

\item[\rm(2)]  We have
$$
\begin{aligned}
\log\C_r\left(\frac x{2}\right)&=-\frac{(r-1)!}{(-2\pi i)^{r-1}}\sum_{k=0}^{r-1}\frac{(-\pi i)^k}{k!}{\rm Li}_{r-k}(-e^{\pi ix})x^k
-\frac{\pi i}{r}\left(\frac x2\right)^{r} \\
&\quad-\frac{(r-1)!}{(-2\pi i)^{r-1}}\zeta_E(r)
\end{aligned}
$$
for ${\rm Im}(x)>0.$
\item[\rm(3)]  For $2\leq r\in2\mathbb Z$ and $0\leq x<1,$ we have
$$
\begin{aligned}
\C_r\left(\frac x{2}\right)&=\left(2\cos\frac{\pi x}{2}\right)^{\left(\frac x2\right)^{r-1}} \\
&\times\exp\biggl((-1)^{\frac{r}{2}}\frac{(r-1)!}{(2\pi)^{r-1}}\sum_{\substack{1\leq k\leq r-3 \\ k\,:{\rm \,odd}}}
\frac{(-1)^{\frac{k-1}2}(\pi x)^k}{k!}\sum_{n=1}^\infty \frac{(-1)^n\cos(\pi nx)}{n^{r-k}} \\
&\quad\qquad-(-1)^{\frac{r}{2}}\frac{(r-1)!}{(2\pi)^{r-1}}\sum_{\substack{0\leq k\leq r-2 \\ k\,:{\rm \,even}}}
\frac{(-1)^{\frac{k}2}(\pi x)^k}{k!}\sum_{n=1}^\infty \frac{(-1)^n\sin(\pi nx)}{n^{r-k}} \biggl).
\end{aligned}
$$

\item[\rm(4)]  For $3\leq r\in1+2\mathbb Z$ and $0\leq x<1,$ we have
$$
\begin{aligned}
\C_r\left(\frac x{2}\right)&=\left(2\cos\frac{\pi x}{2}\right)^{\left(\frac x2\right)^{r-1}} \\
&\times\exp\biggl(-(-1)^{\frac{r-1}{2}}\frac{(r-1)!}{(2\pi)^{r-1}}\sum_{\substack{0\leq k\leq r-3 \\ k\,:{\rm \,even}}}
\frac{(-1)^{\frac k2}(\pi x)^k}{k!}\sum_{n=1}^\infty \frac{(-1)^n\cos(\pi nx)}{n^{r-k}} \\
&\quad\qquad-(-1)^{\frac{r-1}{2}}\frac{(r-1)!}{(2\pi)^{r-1}}\sum_{\substack{1\leq k\leq r-2 \\ k\,:{\rm \,odd}}}
\frac{(-1)^{\frac{k-1}2}(\pi x)^k}{k!}\sum_{n=1}^\infty \frac{(-1)^n\sin(\pi nx)}{n^{r-k}} \\
&\quad\qquad -(-1)^{\frac{r-1}2}\frac{(r-1)!}{(2\pi)^{r-1}}\zeta_E(r)\biggl).
\end{aligned}
$$
\end{enumerate}
\end{theorem}

\begin{remark}
The analogue results  for multiple sine function have been proved by  Kurokawa-Koyama \cite[p. 849, Theorem 2.8]{KK}
(also see  Kurokawa \cite[p. 222, Theorem 2]{Ku3}). 
\end{remark}

For $r=2,3,4$ and 5, Theorem \ref{integ-poly} implies the following results.

\begin{corollary}\label{integ-poly-ex}
For $0\leq x<1,$ we have
$$
\begin{aligned}
\C_2\left(\frac x2\right)&=\left(2\cos\frac{\pi x}{2}\right)^{\frac x2}
\exp\left(\frac1{2\pi}\sum_{n=1}^\infty \frac{(-1)^n\sin(\pi nx)}{n^{2}}\right), \\
\C_3\left(\frac x2\right)&=\left(2\cos\frac{\pi x}{2}\right)^{\left(\frac x2\right)^2}
\exp\biggl(\frac1{2\pi^2}\sum_{n=1}^\infty \frac{(-1)^n\cos(\pi nx)}{n^{3}} \\
&\quad +\frac x{2\pi}
\sum_{n=1}^\infty \frac{(-1)^n\sin(\pi nx)}{n^{2}}+\frac1{2\pi^2}\zeta_E(3)\biggl),
\end{aligned}
$$
$$
\begin{aligned}
\C_4\left(\frac x2\right)&=\left(2\cos\frac{\pi x}{2}\right)^{\left(\frac x2\right)^3}
\exp\biggl(\frac{3x}{4\pi^2}\sum_{n=1}^\infty \frac{(-1)^n\cos(\pi nx)}{n^{3}} \\
&\quad -\frac 3{4\pi^3}\sum_{n=1}^\infty \frac{(-1)^n\sin(\pi nx)}{n^{4}}
+\frac{3x^2}{8\pi}\sum_{n=1}^\infty \frac{(-1)^n\sin(\pi nx)}{n^{2}}\biggl), \\
\C_5\left(\frac x2\right)&=\left(2\cos\frac{\pi x}{2}\right)^{\left(\frac x2\right)^4}
\exp\biggl(-\frac{3}{2\pi^4}\sum_{n=1}^\infty \frac{(-1)^n\cos(\pi nx)}{n^{5}} \\
&\quad +\frac{3x^2}{4\pi^2}\sum_{n=1}^\infty \frac{(-1)^n\cos(\pi nx)}{n^{3}}
-\frac{3x}{2\pi^3}\sum_{n=1}^\infty \frac{(-1)^n\sin(\pi nx)}{n^{4}} 
\\
&\quad+\frac{x^3}{4\pi}\sum_{n=1}^\infty \frac{(-1)^n\sin(\pi nx)}{n^{2}}
-\frac3{2\pi^4}\zeta_E(5) \biggl).
\end{aligned}
$$
\end{corollary}
\begin{remark}
In particular, setting $x=\frac12$ in the above relations  and  by using the expansions
$$\sum_{n=1}^\infty \frac{(-1)^n\sin\left(\frac{\pi n}2\right)}{n^{s}}
=-\sum_{n=0}^\infty\frac{(-1)^n}{(2n+1)^s}=-\beta(s)$$
and
$$\sum_{n=1}^\infty \frac{(-1)^n\cos\left(\frac{\pi n}2\right)}{n^{s}}=\sum_{n=1}^\infty\frac{(-1)^n}{(2n)^s}
=-\frac1{2^s}\zeta_E(s),$$
we recover Corollary \ref{rem-ex}.
\end{remark}

By (\ref{lam}), (\ref{di-def}) and (\ref{prin-L}), the Dirichlet series with coefficients $(-1)^n\sin\left(\frac{\pi n}3\right)$ and 
$(-1)^n\cos\left(\frac{\pi n}3\right)$ can be
calculated as follows:
\begin{equation}\label{di-3-sin}
\begin{aligned}
\sum_{n=1}^\infty \frac{(-1)^n\sin\left(\frac{\pi n}3\right)}{n^{s}}
&=\frac{\sqrt3}{2}\left(\sum_{n\equiv1,2\!\!\!\!\!\pmod{6}}\frac{(-1)^n}{n^s} 
-\sum_{n\equiv4,5\!\!\!\!\!\pmod{6}}\frac{(-1)^n}{n^s} \right) \\
&=\frac{\sqrt3}{2}\left(-\left(\sum_{n\equiv1\!\!\!\!\!\pmod{6}}\frac{1}{n^s} 
-\sum_{n\equiv 5\!\!\!\!\!\pmod{6}}\frac{1}{n^s}  \right)\right. \\
&\qquad\qquad+\left.\left( \sum_{n\equiv 2\!\!\!\!\!\pmod{6}}\frac{1}{n^s} -\sum_{n\equiv 4\!\!\!\!\!\pmod{6}}\frac{1}{n^s}
\right)  \right) \\
&=\frac{\sqrt3}{2}\left(-L(s,\chi_6)+ \frac1{2^s}L(s,\chi_3)\right)
\end{aligned}
\end{equation}
and
\begin{equation}\label{di-3-cos}
\begin{aligned}
\sum_{n=1}^\infty \frac{(-1)^n\cos\left(\frac{\pi n}3\right)}{n^{s}}
&=\frac{1}{2}\left(\sum_{n\equiv1,5\!\!\!\!\!\pmod{6}}\frac{(-1)^n}{n^s} 
-\sum_{n\equiv2,4\!\!\!\!\!\pmod{6}}\frac{(-1)^n}{n^s} \right) \\
&\qquad\qquad+ \sum_{n\equiv 0\!\!\!\!\!\pmod{6}}\frac{(-1)^n}{n^s} -\sum_{n\equiv 3\!\!\!\!\!\pmod{6}}\frac{(-1)^n}{n^s} \\
&=-\frac{1}{2}\sum_{n\equiv1,2,4,5\!\!\!\!\!\pmod{6}}\frac{1}{n^s} 
+ \sum_{n\equiv 0\!\!\!\!\!\pmod{6}}\frac{1}{n^s} +\sum_{n\equiv 3\!\!\!\!\!\pmod{6}}\frac{1}{n^s} \\
&=-\frac12 L(s,\mathbbm{1}_3)+\frac1{6^s}\zeta(s)+\frac1{3^s}\lambda(s) \\
&=-\frac12\left(1-\frac1{3^s}\right)\zeta(s)+\frac1{6^s}\zeta(s)+\frac1{3^s}\lambda(s) \\
&=\frac12\left(\frac1{3^{s-1}}-1\right)\zeta(s).
\end{aligned}
\end{equation}

Now, setting $x=\frac13$ in Corollary \ref{integ-poly-ex}, by
(\ref{di-3-sin}) and (\ref{di-3-cos}) we get the following corollary.

\begin{corollary}\label{c(1/6)}
$$
\begin{aligned}
&\C_2\left(\frac16\right)=3^{\frac1{12}}\exp\left(\frac{\sqrt3}{4\pi}\left(\frac14L(2,\chi_3)-L(2,\chi_6)\right)\right), \\
&\C_3\left(\frac16\right)=3^{\frac1{72}}\exp\left(\frac{11}{72\pi^2}\zeta(3)
+\frac{\sqrt3}{12\pi}\left(\frac14L(2,\chi_3)-L(2,\chi_6)\right)\right).
\end{aligned}
$$
\end{corollary}

\begin{remark}\label{1.6}
From Corollary \ref{c(1/6)}, we immediately obtain another expression of $\zeta(3)$:
$$\zeta(3)=\frac{72\pi^2}{11}\log\left(\frac{3^{\frac1{72}}\C_3\left(\frac16\right)}{\C_2\left(\frac16\right)^{\frac13}}\right).$$
\end{remark}

\section{Multiple cosine functions}

In this section, to our purpose,
we state some basic properties of multiple cosine functions. 
Some of them have been reported in \cite[p. 848, Remark 2.7]{KK}, \cite[p. 124]{KW03}
and \cite[Proposition 3.1(4) and 5.1(4)]{KW04}.

First, we prove the following proposition which is necessary to derive several properties of multiple cosine functions. 
Note that it has appeared in \cite[p. 848]{KK} and \cite[p. 124]{KW03} without proof.

\begin{proposition}\label{pro1}
For $r = 2, 3, 4, \ldots,$ we have $\C_r(0)=1$ and $\C_r(x)$ is a meromorphic function in $x\in\mathbb C$ satisfying
$$\frac{\C_r'(x)}{\C_r(x)}=-\pi x^{r-1}\tan(\pi x).$$
Thus we have the integral representation 
$$\C_r(x)=\exp\left(-\int_0^x\pi t^{r-1}\tan(\pi t)dt\right),$$
where the contour lies in $\mathbb C\setminus\{\pm\frac12,\pm\frac32,\ldots\}.$
\end{proposition}
\begin{proof}
The proof goes similarly with \cite[Proposition 1]{KK05} by calculating the logarithmic derivative. 
When $r=1,$ the result follows immediately from (\ref{cos1-ex}). For $r=2,3,4,\ldots,$
by using
$$
\begin{aligned}
\C_r(x)
&=\prod_{n=1,n\text{:odd}}^\infty 
\left\{P_r\left(\frac{x}{\frac n2}\right)P_r\left(-\frac{x}{\frac n2}\right)^{(-1)^{r-1}}\right\}^{\left(\frac n2\right)^{r-1}} \\
&=\prod_{n=1}^\infty 
\left\{P_r\left(\frac{x}{\frac{2n-1}2}\right)P_r\left(-\frac{x}{\frac{2n-1}2}\right)^{(-1)^{r-1}}\right\}^{\left(\frac {2n-1}2\right)^{r-1}}
\end{aligned}
$$
and (\ref{Pr}),
we obtain
$$
\begin{aligned}
\log\C_r(x)&=\sum_{n=1}^\infty\left(\frac{2n-1}2\right)^{r-1}
\left\{\log P_r\left(\frac{x}{\frac{2n-1}2}\right)+(-1)^{r-1}\log P_r\left(-\frac{x}{\frac{2n-1}2}\right)\right\} \\
&=\sum_{n=1}^\infty\left(\frac{2n-1}2\right)^{r-1}\left\{\log\left(1-\frac{x}{\frac{2n-1}2}\right)+(-1)^{r-1} \log\left(1+\frac{x}{\frac{2n-1}2}\right)\right. \\
&\quad+\left( \frac{x}{\frac{2n-1}2}+\frac12\left(\frac{x}{\frac{2n-1}2}\right)^2+\cdots
+\frac1r\left(\frac{x}{\frac{2n-1}2}\right)^r\right) \\
&\quad\left.+(-1)^{r-1}\left( -\frac{x}{\frac{2n-1}2}+\frac12\left(-\frac{x}{\frac{2n-1}2}\right)^2+\cdots
+\frac1r\left(-\frac{x}{\frac{2n-1}2}\right)^r\right)\right\}.
\end{aligned}
$$
Hence
$$
\begin{aligned}
\frac{\C_r'(x)}{\C_r(x)}&=\sum_{n=1}^\infty\left(\frac{2n-1}2\right)^{r-1}
\left\{\frac1{x-\frac{2n-1}2} +\frac{(-1)^{r-1}}{x+\frac{2n-1}2} \right. \\
&\quad+\left(\frac1{\frac{2n-1}2}+\frac x{\left(\frac{2n-1}2\right)^2} +\cdots+\frac{x^{r-1}}{\left(\frac{2n-1}2\right)^r}\right) \\
&\quad+\left.(-1)^{r-1}\left(-\frac1{\frac{2n-1}2}+\frac x{\left(\frac{2n-1}2\right)^2} +\cdots+(-1)^r\frac{x^{r-1}}{\left(\frac{2n-1}2\right)^r}\right)\right\}. \\
\end{aligned}
$$
Then by observing the expressions
$$\frac1{\frac{2n-1}2}+\frac x{\left(\frac{2n-1}2\right)^2} +\cdots+\frac{x^{r-1}}{\left(\frac{2n-1}2\right)^r}
=\frac{\left(\frac{x}{\frac{2n-1}{2}}\right)^r-1}{x-\frac{2n-1}{2}}$$
and
$$-\frac1{\frac{2n-1}2}+\frac x{\left(\frac{2n-1}2\right)^2} +\cdots+(-1)^r\frac{x^{r-1}}{\left(\frac{2n-1}2\right)^r}
=\frac{(-1)^r\left(\frac{x}{\frac{2n-1}{2}}\right)^r-1}{x+\frac{2n-1}{2}},$$
we see that
$$
\begin{aligned}
\frac{\C_r'(x)}{\C_r(x)}&=\sum_{n=1}^\infty
\left(\frac{\frac{x^r}{\frac{2n-1}{2}}}{x-\frac{2n-1}{2}} - \frac{\frac{x^r}{\frac{2n-1}{2}}}{x+\frac{2n-1}{2}}\right) \\
&=\sum_{n=1}^\infty\frac{2x^r}{x^2-\left(\frac{2n-1}{2}\right)^2} \\
&=8x^r\sum_{n=1}^\infty\frac1{(2x)^2-(2n-1)^2} \\
&=-\pi x^{r-1}\tan(\pi x),
\end{aligned}
$$
where we have used the expansion \cite[p. 43, 1.421(1)]{GR}
$$\tan\left(\frac{\pi x}2\right)=\frac{4x}{\pi}\sum_{n=1}^\infty\frac1{(2n-1)^2-x^2}.$$
This completes the proof of Proposition \ref{pro1}.
\end{proof}

From Proposition \ref{pro1}, we get a new proof for the following result by Kurokawa and Wakayama \cite[p. 125]{KW03}.

\begin{corollary}[Duplication formulas]\label{dup}
For $r = 1,2, 3,  \ldots,$ 
$$\C_r(x)^{2^{r-1}}=\frac{\mathcal S_r(2x)}{\mathcal S_r(x)^{2^{r-1}}}.$$
\end{corollary}
\begin{proof}
By Proposition \ref{pro1} and the trigonometric identity
$$\cot(x+y)=\frac{\cot x\cot y-1}{\cot x+\cot y},$$
we have
$$
\begin{aligned}
2^{r-1}\log\C_r(x)&=-2^{r-1}\int_0^x\pi t^{r-1}\tan(\pi t)dt \\
&=2^{r-1}\int_0^x\pi t^{r-1}\left(\frac{\cot^2(\pi t)-1}{\cot(\pi t)}-\cot(\pi t)\right) dt \\
&=2^{r-1}\int_0^x\pi t^{r-1}(2\cot(2\pi t)-\cot(\pi t))dt \\
&=\int_0^x\pi (2t)^{r-1}\cot(2\pi t)d(2t)-2^{r-1}\int_0^x\pi t^{r-1}\cot(\pi t)dt \\
&=\int_0^{2x}\pi t^{r-1}\cot(\pi t)dt-2^{r-1}\int_0^x\pi t^{r-1}\cot(\pi t)dt \\
&=\log \mathcal S_r(2x)-2^{r-1}\log \mathcal S_r(x) \\
&=\log\frac{\mathcal S_r(2x)}{\mathcal S_r(x)^{2^{r-1}}},
\end{aligned}
$$
where we have used the identity $\log \mathcal S_r(x)=\int_0^x\pi t^{r-1}\cot(\pi t)dt$ 
(see \cite[Proposition 2]{KK05}).
Thus the corollary follows.
\end{proof}

\begin{proposition}\label{pro2}
For $0\leq x<1$ and $r = 2, 3, 4, \ldots,$ 
$$\log\C_r\left(\frac x2\right)=-\frac{1}{2^r}\int_0^x \pi t^{r-1}\tan\left(\frac{\pi t}{2}\right)dt.$$
\end{proposition}
\begin{proof}
Since $\C_r(0)=1,$ both sides of the above equation are 0 at the boundary point $x=0.$ So it only needs to show that the logarithmic differentiations of both sides are equal,
but this immediately  follows from Proposition \ref{pro1} by  replacing $x$ with $\frac x2$.
\end{proof}

\begin{proposition}\label{pro1-0}
Let $r=2,3,4,\ldots.$ Then
\begin{enumerate}
\item[\rm(1)]  We have
$$\C_r(x+1)=\frac{\C_r(1)}{2}\prod_{k=1}^r\C_k(x)^{\binom{r-1}{k-1}}.$$
\item[\rm(2)]  For $3\leq N\in1+2\mathbb Z,$ we have
$$
\begin{aligned}
\C_r(Nx)&=A_r(N)\prod_{a=0}^{N-1}\C_r\left(x+\frac{2a}N\right)^{N^{r-1}} \\
&\quad\times\prod_{k=1}^{r-1}\prod_{a=1}^{N-1}\C_r\left(x+\frac{2a}N\right)^{(-1)^{r-k}\binom{r-1}{k-1}(2a)^{r-k}N^{k-1}},
\end{aligned}
$$
where
$$
\begin{aligned}
A_r(N)^{-1}&=\left(\C_r\left(\frac{2}N\right)\cdots\C_r\left(\frac{2(N-1)}N\right)\right)^{N^{r-1}} \\
&\quad\times\prod_{k=1}^{r-1}\prod_{a=1}^{N-1}\C_r\left(\frac{2a}N\right)^{(-1)^{r-k}\binom{r-1}{k-1}(2a)^{r-k}N^{k-1}}.
\end{aligned}
$$
\end{enumerate}
\end{proposition}
\begin{proof}
We employ the method in \cite[Theorem 2.10(a) and (b)]{KK}. By Proposition \ref{pro1}, we have
$$
\begin{aligned}
\frac{\C_r'(x+1)}{\C_r(x+1)}&=-\pi(x+1)^{r-1}\tan(\pi x) \\
&=-\pi\sum_{k=1}^r\binom{r-1}{k-1}x^{k-1}\tan(\pi x) \\
&=\sum_{k=1}^r\binom{r-1}{k-1}\frac{\C_k'(x)}{\C_k(x)},
\end{aligned}
$$
thus
\begin{equation}\label{d-Cr} 
\begin{aligned}
\frac{d}{dx}\log \C_r(x+1)&=\sum_{k=1}^r\binom{r-1}{k-1}\frac d{dx}\log\C_k(x) \\
&=\frac d{dx}\log\prod_{k=1}^r\C_k(x)^{\binom{r-1}{k-1}}.
\end{aligned}
\end{equation}
Therefore  
\begin{equation}\label{A} 
\C_r(x+1)=C\prod_{k=1}^r\C_k(x)^{\binom{r-1}{k-1}}
\end{equation}
with some constant $C.$
Now put
\begin{equation}\label{B}
F(x)=\frac{\C_r(x+1)}{\prod_{k=1}^r\C_k(x)^{\binom{r-1}{k-1}}}.
\end{equation}
Since $\C_1(0)=2$ and $\C_2(0)=\cdots=\C_r(0)=1,$ by (\ref{A}) and (\ref{B})  we have
$$C=F(0)=\lim_{x\to0}\frac{\C_r(x+1)}{\prod_{k=1}^r\C_k(x)^{\binom{r-1}{k-1}}}=\frac{\C_r(1)}{2}$$
and (1) follows.

For (2), let $3\leq N\in1+2\mathbb Z.$  By
\begin{equation}\label{cos-mul}
2\cos(Nx)=\prod_{a=0}^{N-1}\left\{2\cos\left(x+\frac{2\pi a}{N}\right)\right\}
\end{equation}
 (see \cite[p. 41, 1.393(1)]{GR}), we have
\begin{equation}\label{tan-eq} 
\begin{aligned}
N\tan(\pi Nx)&=-\frac1\pi \frac d{dx}\log2\cos(\pi Nx) \\
&=-\sum_{a=0}^{N-1}\frac1\pi\frac d{dx}\log 2\cos\left(\pi x+\frac{2\pi a}{N}\right) \\
&=\sum_{a=0}^{N-1}\tan\pi\left(x+\frac{2a}{N}\right).
\end{aligned}
\end{equation}
Then combing Proposition \ref{pro1} and (\ref{tan-eq}), we get
\begin{equation}\label{Cr-tan}
\begin{aligned}
\frac d{dx}\log\C_r(Nx)&=-\pi N(Nx)^{r-1}\tan(\pi Nx) \\
&=-\pi N^{r-1}x^{r-1}\sum_{a=0}^{N-1}\tan\pi\left(x+\frac{2a}{N}\right).
\end{aligned}
\end{equation}
Since 
$$x^{r-1}=\left(x+\frac{2a}{N}\right)^{r-1}+\sum_{k=1}^{r-1}\binom{r-1}{k-1}\left(-\frac{2a}{N}\right)^{r-k}
\left(x+\frac{2a}{N}\right)^{k-1},$$ 
by applying Proposition \ref{pro1} again we obtain
\begin{equation}\label{Cr-tan2}
\begin{aligned}
\frac d{dx}\log\C_r(Nx)
&=-\pi N^{r-1}\left(x+\frac{2a}{N}\right)^{r-1}\sum_{a=0}^{N-1}\tan\pi\left(x+\frac{2a}{N}\right) \\
&\quad-\pi N^{r-1}\sum_{k=1}^{r-1}\binom{r-1}{k-1}\left(-\frac{2a}{N}\right)^{r-k} \left(x+\frac{2a}{N}\right) ^{k-1}
\sum_{a=0}^{N-1}\tan\pi\left(x+\frac{2a}{N}\right)\\
&=N^{r-1}\sum_{a=0}^{N-1}\frac{d}{dx}\left(\log\C_r \left(x+\frac{2a}{N}\right) \right. \\
&\left.\quad\qquad\qquad\qquad
+\sum_{k=1}^{r-1}(-1)^{r-k}\binom{r-1}{k-1}\left(\frac{2a}{N}\right)^{r-k}\log\C_r \left(x+\frac{2a}{N}\right) 
\right)
\\
&=N^{r-1}\sum_{a=0}^{N-1}\frac{d}{dx}\left(\log\C_r \left(x+\frac{2a}{N}\right) \right) \\
&\quad
+N^{r-1}\sum_{a=0}^{N-1}
\sum_{k=1}^{r-1}(-1)^{r-k}\binom{r-1}{k-1}\left(\frac{2a}{N}\right)^{r-k}\frac{d}{dx}\left(\log\C_r \left(x+\frac{2a}{N}\right) 
\right)
\\
&=N^{r-1}\frac d{dx}\log\left(\prod_{a=0}^{N-1}\C_r\left(x+\frac{2a}{N}\right)\right) \\
&\quad+N^{r-1}\frac d{dx}\log\left(\prod_{a=1}^{N-1}\prod_{k=1}^{r-1}
\C_r\left(x+\frac{2a}{N}\right)^{(-1)^{r-k}\binom{r-1}{k-1}\left(\frac{2a}{N}\right)^{r-k}}\right),
\end{aligned}
\end{equation}
which leads to (2).
\end{proof}

\begin{proposition}\label{pro1-1}
For $r = 2, 3, 4, \ldots,$ the multiple cosine function $\C_r(x)$ satisfies the following second order algebraic differential equation
$$\C_r''(x)=(1-x^{1-r})\frac{\C_r'(x)^2}{\C_r(x)}+(r-1)\frac{\C_r'(x)}{x}-\pi^2x^{r-1}\C_r(x)$$
with $\C_r(0)=1$ and $\C_r'(0)=0.$ 
\end{proposition}
\begin{proof}
From Proposition \ref{pro1}, we obtain
\begin{equation}\label{E}
\begin{aligned}
\frac{d}{dx}\left(\frac1{\pi x^{r-1}}\frac{\C_r'(x)}{\C_r(x)}\right)&=-\pi\sec^2(\pi x) \\
&=-\pi(\tan^2(\pi x)+1) \\
&=-\pi\left(\left(-\frac1{\pi x^{r-1}}\frac{\C_r'(x)}{\C_r(x)}\right)^2+1\right).
\end{aligned}
\end{equation}
On the other hand, by applying the derivative formula in calculus directly, we have
\begin{equation}\label{F}
\frac{d}{dx}\left(\frac1{\pi x^{r-1}}\frac{\C_r'(x)}{\C_r(x)}\right)=-\frac{r-1}{\pi x^{r}}\frac{\C_r'(x)}{\C_r(x)}
+\frac1{\pi x^{r-1}}\left(\frac{\C_r''(x)}{\C_r(x)}-\frac{\C_r'(x)^2}{\C_r(x)^2}\right).
\end{equation}
Then  by comparing (\ref{E}) and (\ref{F}), we get
$$-\frac{r-1}x\C_r'(x)+\C_r''(x)-\frac{\C_r'(x)^2}{\C_r(x)}
=-\frac1{x^{r-1}}\frac{\C_r'(x)^2}{\C_r(x)}-\pi^2x^{r-1}\C_r(x),$$
which is equivalent to the statement of the proposition.
\end{proof}

\begin{remark}
Propositions \ref{pro1-1} is an analogy of Painlev\'e's differential equation of type III. 
\end{remark}

\section{Proofs of the results}\label{proofs}

In this section, we prove the results stated in Section \ref{main results}.

\subsection*{The first proof of Theorem \ref{thm1}.}
As remarked by Allouche in an email to us, this result can be implied by (\ref{K-K})
if using $$\sin \theta=2\sin \frac \theta2 \cos \frac\theta2$$  to write 
\begin{equation}\label{cos1}
\log\left(\cos\frac\theta 2\right)=\log(\sin\theta)-\log\left(\sin\frac\theta 2\right)-\log2.
\end{equation}
 and noticing the relation between $\C_{r}(x)$ and $\mathcal S_{r}(x)$ (see (\ref{sin-cos})). 
 Following his idea, we give a detailed proof as follows. 
 
From (\ref{cos1}) and (\ref{K-K}) we have
\begin{equation}\label{thm-e1}
\begin{aligned}
\int_0^x\theta^{r-2}\log\left(\cos\frac\theta 2\right)d\theta
&=\int_0^x\theta^{r-2}\log(\sin\theta)d\theta \\
&\quad-2^{r-1}\int_0^{\frac x2}\theta^{r-2}\log(\sin\theta)d\theta -\log2\int_0^x\theta^{r-2} d\theta \\
&=\frac{x^{r-1}}{r-1}\log(\sin x)-\frac{\pi^{r-1}}{r-1}\log\mathcal S_r\left(\frac x{\pi}\right) \\
&\quad-2^{r-1}\left(\frac{\left(\frac x2\right)^{r-1}}{r-1}\log\left(\sin \frac x2\right)
-\frac{\pi^{r-1}}{r-1}\log\mathcal S_r\left(\frac x{2\pi}\right)\right) \\
&\quad-\frac{x^{r-1}}{r-1}\log2\\
&=\frac{x^{r-1}}{r-1}\left( \log(\sin x)-\log\left(\sin \frac x2\right)-\log2\right) \\
&\quad-\frac{\pi^{r-1}}{r-1}\left(\log\mathcal S_r\left(\frac x{\pi}\right)-2^{r-1}\log\mathcal S_r\left(\frac x{2\pi}\right) \right).
\end{aligned}
\end{equation}
On the other hand, the logarithmic of (\ref{sin-cos}) yields
\begin{equation}\label{du-form}
2^{r-1}\log\mathcal C_r\left(\frac x{2\pi}\right)=\log\mathcal S_r\left(\frac x{\pi}\right)
-2^{r-1}\log\mathcal S_r\left(\frac x{2\pi}\right).
\end{equation}
Then substituting  (\ref{cos1}) and (\ref{du-form}) into (\ref{thm-e1}), we get our result.

\subsection*{The second proof of Theorem \ref{thm1}.}
Here we also derive this result directly.

From Proposition \ref{pro2} and the integration by parts, we have
$$
\begin{aligned}
\log\C_r\left(\frac x2\right)&=\frac1{2^{r-1}}\left(\left[t^{r-1}\log\left(\cos\frac{\pi t}{2}\right)\right]_0^x 
- \int_0^x (r-1)t^{r-2}\log\left(\cos\frac{\pi t}{2}\right)dt\right) \\
&=\frac1{2^{r-1}}\left(x^{r-1}\log\left(\cos\frac{\pi x}{2}\right)
-(r-1) \int_0^x t^{r-2}\log\left(\cos\frac{\pi t}{2}\right)dt\right).
\end{aligned}
$$
Then changing the variable to $\theta=\pi t$ in this integral, we have
$$
\log\C_r\left(\frac x2\right)
=\frac1{2^{r-1}}\left(x^{r-1}\log\left(\cos\frac{\pi x}{2}\right)
-\frac{r-1}{\pi^{r-1}} \int_0^{\pi x} \theta^{r-2}\log\left(\cos\frac{\theta}{2}\right)d\theta\right).
$$
Now, letting $x\to\frac x\pi,$ the assertion follows.

\subsection*{Proof of Theorem \ref{thm1-gen}.}

To prove this, we need the following two lemmas.

\begin{lemma}\label{lem1-gen} For $r=0,1,2,\ldots,$ we have
$$\begin{aligned}
\int_0^{x}\theta^r\cos(n\theta)d\theta&=\sum_{k=0}^r\binom rk\frac{k!}{n^{k+1}}\sin\left(nx+\frac{k\pi}{2}\right)x^{r-k}
-\frac{r!}{n^{r+1}}\sin\left(\frac{r\pi}{2}\right).
\end{aligned}$$
\end{lemma}
\begin{proof}
Recall the indefinite integral \cite[p. 226, 2.633(2)]{GR}
$$\int\theta^r\cos(n\theta)d\theta=\sum_{k=0}^rk!\binom rk\frac{\theta^{r-k}}{n^{k+1}}\sin\left(n\theta+\frac{k\pi}{2}\right)+C.$$
Hence, in view of the relation $0^{r-k}=0$ if $k<r$ and 1 if $k=r,$ we have
$$\begin{aligned}
\int_0^{x}\theta^r\cos(n\theta)d\theta&=\sum_{k=0}^{r-1}k!\binom rk\frac1{n^{k+1}}
\left[ \theta^{r-k}\sin\left(n\theta+\frac{k\pi}{2}\right)\right]_0^{x} \\
&\quad+r!\frac1{n^{r+1}}
\left[ \sin\left(n\theta+\frac{r\pi}{2}\right)\right]_0^{x} \\
&=\sum_{k=0}^{r-1}k!\binom rk\frac1{n^{k+1}}
x^{r-k}\sin\left(nx+\frac{k\pi}{2}\right) \\
&\quad+r!\frac1{n^{r+1}}\left(
\sin\left(nx+\frac{r\pi}{2}\right)-\sin\left(\frac{r\pi}{2}\right)\right) \\
&=\sum_{k=0}^{r}k!\binom rk\frac1{n^{k+1}}
x^{r-k}\sin\left(nx+\frac{k\pi}{2}\right) \\
&\quad-r!\frac1{n^{r+1}}\sin\left(\frac{r\pi}{2}\right),
\end{aligned}$$
which completes the proof of Lemma \ref{lem1-gen}.
\end{proof}

\begin{lemma}\label{lem2-gen} For $r=0,1,2,\ldots,$ we have
$$\begin{aligned}
\sum_{n=1}^\infty\frac{(-1)^{n-1}}{n}\int_0^{\frac\pi 2}\theta^r\cos(n\theta)d\theta
&=\sum_{k=0}^{\left\lfloor \frac{r}{2}\right\rfloor }(-1)^k(2k)!\binom{r}{2k}
\left(\frac\pi2\right)^{r-2k} \\
&\quad\times\beta(2k+2) \\
&\quad+\sum_{k=1}^{\left\lceil \frac{r}2\right\rceil}\frac{(-1)^{k-1}(2k-1)!}{2^{2k+1}}\binom{r}{2k-1}
\left(\frac\pi2\right)^{r-2k+1} \\
&\quad\times\zeta_E(2k+1) \\
&\quad-r!\sin\left(\frac{r\pi}{2}\right)\zeta_E(r+2),
\end{aligned}$$
where $\lfloor x\rfloor =\max\{m\in \mathbb {Z} \mid m\leq x\}$ and $\lceil x\rceil =\min\{m\in \mathbb {Z} \mid m\geq x\}.$
\end{lemma}
\begin{proof}
Setting $x=\frac\pi2$ in Lemma \ref{lem1-gen}, by the fundamental formula of angle addition for the sine function, we obtain
\begin{equation}\label{lem1}
\begin{aligned}
\sum_{n=1}^\infty\frac{(-1)^{n-1}}{n}\int_0^{\frac\pi 2}\theta^r\cos(n\theta)d\theta
&=\sum_{n=1}^\infty\frac{(-1)^{n-1}}{n}\left(
\sum_{k=0}^rk!\binom rk\left(\frac\pi2\right)^{r-k}\frac1{n^{k+1}} \right. \\
&\quad\times\left(\sin\left(\frac{n\pi}{2}\right)\cos\left(\frac{k\pi}{2}\right)
+\sin\left(\frac{k\pi}{2}\right)\cos\left(\frac{n\pi}{2}\right)\right) \\
&\left.\quad-r!\frac1{n^{r+1}}\sin\left(\frac{r\pi}{2}\right)\right).
\end{aligned}
\end{equation}
For the calculation of the right hand side,  we split the summation into three parts $\textbf{I}_1,\textbf{I}_2,\textbf{I}_3$ according to the terms
$$\sin\left(\frac{n\pi}{2}\right)\cos\left(\frac{k\pi}{2}\right),~\sin\left(\frac{k\pi}{2}\right)\cos\left(\frac{n\pi}{2}\right)
\text{ and }\sin\left(\frac{r\pi}{2}\right).$$
First we calculate the sum $\textbf{I}_1.$ From (\ref{beta-def}), we have
\begin{equation}\label{lem2}
\begin{aligned}
\textbf{I}_1&=\sum_{n=1}^\infty\frac{(-1)^{n-1}}{n}
\sum_{k=0}^rk!\binom rk\left(\frac\pi2\right)^{r-k}\frac1{n^{k+1}}
\sin\left(\frac{n\pi}{2}\right)\cos\left(\frac{k\pi}{2}\right) \\
&=\sum_{k=0}^rk!\binom rk\left(\frac\pi2\right)^{r-k}\cos\left(\frac{k\pi}{2}\right)
\sum_{n=1}^\infty\frac{(-1)^{n-1}}{n}\frac1{n^{k+1}}\sin\left(\frac{n\pi}{2}\right) \\
&=\sum_{k=0}^{\left\lfloor \frac{r}{2}\right\rfloor}(2k)!\binom r{2k}\left(\frac\pi2\right)^{r-2k}(-1)^k
\sum_{n=1}^\infty\frac{(-1)^{n-1}}{(2n-1)^{2k+2}} \\
&=\sum_{k=0}^{\left\lfloor \frac{r}{2}\right\rfloor}(-1)^k(2k)!\binom r{2k}\left(\frac\pi2\right)^{r-2k}
\beta(2k+2).
\end{aligned}
\end{equation}
Then we calculate the sum $\textbf{I}_2.$  From (\ref{A-zeta}), we have
\begin{equation}\label{lem3}
\begin{aligned}
\textbf{I}_2&=\sum_{n=1}^\infty\frac{(-1)^{n-1}}{n}
\sum_{k=0}^rk!\binom rk\left(\frac\pi2\right)^{r-k}\frac1{n^{k+1}}
\sin\left(\frac{k\pi}{2}\right)\cos\left(\frac{n\pi}{2}\right) \\
&=\sum_{k=0}^rk!\binom rk\left(\frac\pi2\right)^{r-k}\sin\left(\frac{k\pi}{2}\right)
\sum_{n=1}^\infty\frac{(-1)^{n-1}}{n}\frac1{n^{k+1}}\cos\left(\frac{n\pi}{2}\right) \\
&=\sum_{k=1}^{\left\lceil \frac r2\right\rceil }(2k-1)!\binom r{2k-1}\left(\frac\pi2\right)^{r-2k+1}(-1)^{k-1}
\frac1{2^{2k+1}}\sum_{n=1}^\infty\frac{(-1)^{n-1}}{n^{2k+1}} \\
&=\sum_{k=1}^{\left\lceil \frac r2\right\rceil}\frac{(-1)^{k-1}(2k-1)!}{2^{2k+1}}\binom r{2k-1}\left(\frac\pi2\right)^{r-2k+1}
\zeta_E(2k+1).
\end{aligned}
\end{equation}
Finally, (\ref{A-zeta}) also implies that
\begin{equation}\label{lem4}
\begin{aligned}
\textbf{I}_3&=\sum_{n=1}^\infty\frac{(-1)^{n-1}}{n}\left(-r!\frac1{n^{r+1}}\right)\sin\left(\frac{r\pi}{2}\right) \\
&=-r!\sin\left(\frac{r\pi}{2}\right)\sum_{n=1}^\infty\frac{(-1)^{n-1}}{n^{r+2}} \\
&=-r!\sin\left(\frac{r\pi}{2}\right)\zeta_E(r+2).
\end{aligned}
\end{equation}
Substituting  (\ref{lem2}), (\ref{lem3}) and (\ref{lem4}) into (\ref{lem1}) we get the lemma.
\end{proof}

Now we go to the proof of Theorem \ref{thm1-gen}.
From the following series expansion (see \cite[p. 148]{To})
$$\log\left(\cos\frac \theta2\right)=-\log2+\sum_{n=1}^\infty(-1)^{n-1}\frac{\cos(n\theta)}{n},$$
we have
$$
\begin{aligned}
\int_0^{\frac\pi 2}\theta^{r-2}\log\left(\cos\frac \theta2\right)d\theta
&=\int_0^{\frac\pi 2}\theta^{r-2}\left(-\log2+\sum_{n=1}^\infty(-1)^{n-1}\frac{\cos(n\theta)}{n}\right)d\theta \\
&=-\frac{\log2}{r-1}\left(\frac\pi2\right)^{r-1}
+\sum_{n=1}^\infty\frac{(-1)^{n-1}}{n}\int_0^{\frac\pi 2}\theta^{r-2}\cos(n\theta)d\theta
\end{aligned}
$$
 for $r=2,3,4,\ldots.$ 
Then replacing $r$ by $r-2$ in Lemma \ref{lem2-gen} and substituting the result into the right hand side of the above equation, after some elementary calculations, we obtain the desired result.

\subsection*{Proof of Theorem \ref{integ-2}.}

From Euler's infinite product representation of the cosine function (\ref{cose}), we have
$$\log\left(\cos\theta\right)=\sum_{m=0}^\infty\log\left(1-\frac{4\theta^2}{(2m+1)^2\pi^2}\right).$$
Since 
$$\log(1-\theta)=-\sum_{n=1}^\infty\frac{\theta^n}{n}$$
for $|\theta|<1,$ we see that 
$$
\begin{aligned}
\log\left(\cos\theta\right)&=-\sum_{n=1}^\infty\left(\frac{2\theta}{\pi}\right)^{2n}\frac1n\sum_{m=0}^\infty\frac1{(2m+1)^{2n}} \\
&=-\sum_{n=1}^\infty\frac{2^{2n}\lambda(2n)}{\pi^{2n}n}\theta^{2n}
\end{aligned}
$$
for $|\theta|<\frac\pi2.$ In the above equation, replacing $\theta$ by $\frac\theta2,$ then multiplying both sides by $\theta^{r-2}$ and integrating the result  from 0 to $x,$ we have
\begin{equation}\label{C}
\int_0^x\theta^{r-2}\log\left(\cos \frac\theta2\right)d\theta
=-\sum_{n=1}^\infty\frac{\lambda(2n)}{\pi^{2n}n}\frac{x^{2n+r-1}}{2n+r-1},
\end{equation}
where $0\leq x<\pi.$
On the other hand, by Theorem \ref{thm1} we have
\begin{equation}\label{D}
\int_0^x \theta^{r-2}\log\left(\cos\frac\theta2\right)d\theta=\frac{x^{r-1}}{r-1}\log\left(\cos\frac x2\right)
-\frac{(2\pi)^{r-1}}{r-1}\log\C_r\left(\frac x{2\pi}\right).
\end{equation}
Comparing (\ref{C}) and (\ref{D}) we get
$$ 
\frac{(2\pi)^{r-1}}{r-1}\log\C_r\left(\frac x{2\pi}\right)
=\frac{x^{r-1}}{r-1}\log\left(\cos\frac x2\right)
+\sum_{n=1}^\infty\frac{\lambda(2n)}{\pi^{2n}n}\frac{x^{2n+r-1}}{2n+r-1}.$$
The result is now easily established.

\subsection*{Proof of Theorem \ref{integ-poly}.}

(1) By Proposition \ref{pro2}, we have
$$\log\C_r\left(\frac x2\right)=-\frac{\pi}{2^r}\int_0^x t^{r-1}\tan\left(\frac{\pi t}{2}\right)dt.$$
If ${\rm Im}(x)<0,$  then  by changing the variable $t=x\theta~ (0\leq\theta\leq1)$
and using the following formulas (cf. \cite[p. 223]{Ku3} and \cite[p. 849]{KK})
$$
\begin{aligned}
\tan\left(\frac{\pi x\theta}2\right)
&=\frac1i\left(\frac{1-e^{-i\pi x\theta}}{1+e^{-i\pi x\theta}}\right) =-i\left(-1+\frac{2}{1+e^{-i\pi x\theta}}\right) \\
&=i\left(1+2\sum_{n=0}^\infty(-1)^{n-1}e^{-\pi i nx\theta}\right) \\
&=i\left(-1+2\sum_{n=1}^\infty(-1)^{n-1}e^{-\pi i nx\theta}\right)
\end{aligned}
$$
for $\theta>0$ and
$$
\int_0^1\theta^{r-1}e^{\alpha \theta}d\theta
=(-1)^{r-1}(r-1)!\frac{e^\alpha}{\alpha^r}\left(\sum_{k=0}^{r-1}\frac{(-1)^k}{k!}\alpha^k-e^{-\alpha}\right)
$$
for $\alpha\in\mathbb C\setminus\{0\}$ (see \cite[p. 223]{Ku3} and \cite[p. 850]{KK}),
we get
\begin{equation}\label{2.12}
\begin{aligned}
\log\C_r\left(\frac x2\right)
&=-\frac{i\pi x^r}{2^r}\int_0^1 \theta^{r-1} \left(-1+2\sum_{n=1}^\infty(-1)^{n-1}e^{-\pi i nx\theta}\right) d\theta \\
&=\frac{i\pi x^r}{r2^r}-\frac{i\pi x^r}{2^{r-1}}\sum_{n=1}^\infty(-1)^{n-1}\int_0^1\theta^{r-1}e^{-\pi inx\theta}d\theta \\
&=-\frac{(r-1)!}{(2\pi i)^{r-1}}\sum_{k=0}^{r-1}\frac{(\pi i)^k}{k!}{\rm Li}_{r-k}(-e^{-\pi ix})x^k
+\frac{\pi i}{r2^r}x^r \\
&\qquad\qquad-\frac{(r-1)!}{(2\pi i)^{r-1}}\zeta_E(r),
\end{aligned}
\end{equation}
and (1) is proved.

(2) When ${\rm Im}(x)>0,$ the proof for (2) similar. 

(3) and (4) From (\ref{2.12}) we have
 $$
\begin{aligned}
\log\C_r\left(\frac x2\right)
&=-\frac{(r-1)!}{(2\pi i)^{r-1}}\sum_{k=0}^{r-2}\frac{(\pi i)^k}{k!}{\rm Li}_{r-k}(-e^{-\pi ix})x^k \\
&\quad-\left(\frac{x}{2}\right)^{r-1}{\rm Li}_{1}(-e^{-\pi ix})+\frac{\pi i}{r2^r}x^r-\frac{(r-1)!}{(2\pi i)^{r-1}}\zeta_E(r).
\end{aligned}
$$
Since
$${\rm Li}_{1}(-e^{-\pi ix})=-\log(1+e^{-\pi ix})=\log\left(2e^{-\frac{i\pi x}{2}}\cos\frac{\pi x}2\right),$$
we get
$$
\begin{aligned}
\log\C_r\left(\frac x2\right)
&=-\frac{(r-1)!}{(2\pi i)^{r-1}}\sum_{k=0}^{r-2}\frac{(\pi i)^k}{k!}{\rm Li}_{r-k}(-e^{-\pi ix})x^k \\
&\quad-i\pi\left(\frac{x}{2}\right)^{r} +\log\left(2\cos\frac{\pi x}2\right)^{\left(\frac{x}{2}\right)^{r-1}} \\
&\quad+\frac{\pi i}{r2^r}x^r-\frac{(r-1)!}{(2\pi i)^{r-1}}\zeta_E(r).
\end{aligned}
$$
Then by taking the exponential on the both sides of the above equation, we get
\begin{equation}\label{Cr(x)}
\begin{aligned}
\C_r\left(\frac x{2}\right)&=\left(2\cos\frac{\pi x}{2}\right)^{\left(\frac x2\right)^{r-1}}
\exp\left(-\frac{(r-1)!}{(2\pi i)^{r-1}}\sum_{k=0}^{r-2}\frac{(\pi i x)^k}{k!}{\rm Li}_{r-k}(-e^{-\pi ix})
\right. \\
&\quad \left.-i\pi\left(\frac x2\right)^r+\frac{\pi i}{r}\left(\frac x2\right)^{r}-\frac{(r-1)!}{(2\pi i)^{r-1}}\zeta_E(r)\right).
\end{aligned}
\end{equation}
Finally, by taking the real part in the above expression for $2\leq r\in2\mathbb Z$ and  $3\leq r\in1+2\mathbb Z$ repectively,
also notice that (see (\ref{poly-def}))
$${\rm Li}_{r-k}(-e^{-\pi ix})=\sum_{n=1}^{\infty}\frac{(-1)^{n}e^{-\pi i xn}}{n^{r-k}}=\sum_{n=1}^{\infty}(-1)^{n}\frac{\cos(\pi ixn)-i\sin(\pi i xn)}{n^{r-k}},$$
we obtain (3) and (4). 

\section{Miscellaneous results}

In this section, we present several new representations for $\log\C_r(x)$ and some series involving $\lambda(2k)$, the special values
of Dirichlet's lambda function at positive even integer arguments. 

Let $\Cl_2(\theta)$ be the Clausen function defined by
\begin{equation}\label{cla-def}
\Cl_2(\theta)=\sum_{n=1}^\infty\frac{\sin(n\theta)}{n^2}.
\end{equation}
The Clausen function $\Cl_2$ is related to the following expression (see for instance \cite[p. 106, (2)]{SC})
\begin{equation}\label{cla-poly}
\Cl_2(\theta)=\theta\log\pi -\theta\log\left(\sin\frac\theta2\right)+2\pi
\log\frac{G \left(1-\frac\theta{2\pi}\right)}{G \left(1+\frac\theta{2\pi}\right)},
\end{equation}
where $G(x)$ is the Barnes $G$-function.
From Corollary \ref{integ-poly-ex} and (\ref{cla-poly}), we obtain
\begin{equation}\label{C2-Cl}
\begin{aligned}
\C_2\left(\frac x2\right)&=\left(2\cos\frac{\pi x}{2}\right)^{\frac x2}
\exp\left(\frac1{2\pi}\sum_{n=1}^\infty \frac{\sin(\pi n(x+1))}{n^{2}}\right) \\
&=\left(2\cos\frac{\pi x}{2}\right)^{\frac x2}
\exp\left(\frac1{2\pi}\Cl_2(\pi(x+1))\right)
\end{aligned}
\end{equation}
since $(-1)^n\sin(\pi nx)=\sin(\pi n(x+1))$ for $n=1,2,3,\ldots.$
Taking logarithm on the both sides of (\ref{C2-Cl}) and using (\ref{cla-poly}) with $\theta=\pi(x+1),$ we get
\begin{equation}\label{log-C2-Cl}
\log\C_2\left(\frac x2\right)
=\frac x2 \log(2\pi)+\log\sqrt\pi-\frac12\log\left(\cos\frac{\pi x}{2}\right)
+\log\frac{G \left(\frac12-\frac x{2}\right)}{G \left(\frac32+\frac x2\right)}.
\end{equation}
Then by using  Proposition \ref{pro2} we have
\begin{equation}\label{log-C2-Cl-2}
\int_0^{x}\pi t \tan\left(\frac{\pi t}{2}\right)dt
=-2x\log(2\pi)-2\log\pi+2\log\left(\cos\frac{\pi x}{2}\right)
-4\log\frac{G \left(\frac12-\frac x{2}\right)}{G \left(\frac32+\frac x2\right)}.
\end{equation}
As an application, setting $x=\frac12$ in (\ref{log-C2-Cl}) we get
\begin{equation}\label{log-C2-Cl-ex}
\log\C_2\left(\frac14\right)
=\frac{\log(4\pi)}4 +\log\sqrt\pi
+\log\frac{G \left(\frac14\right)}{G \left(\frac74\right)}.
\end{equation}
By using Corollary \ref{rem-ex}  and (\ref{log-C2-Cl-ex}), we see that
\begin{equation}\label{log-C2-Cl-ex2}
\log\frac{G \left(\frac14\right)}{G \left(\frac74\right)}=-\frac{3\log2}8 -\frac{3\log \pi}4 -\frac G{2\pi}
\end{equation}
which is equivalent to
\begin{equation}\label{log-C2-Cl-ex2-eq}
G \left(\frac74\right)=2^{\frac38} \pi^{\frac34}e^{\frac G{2\pi}}G \left(\frac14\right).
\end{equation}
Then by considering the following expression due to Choi and Srivastava \cite[p. 30, (23)]{SC}:
\begin{equation}\label{CS-1/4}
\log G\left(\frac14\right)=-\frac G{4\pi}-\frac34\log\Gamma\left(\frac14\right)-\frac{9\log A}8+\frac3{32},
\end{equation}
we get
\begin{equation}\label{SK7/4}
G\left(\frac74\right)=2^{\frac38} \pi^{\frac34}e^{\frac G{4\pi}+\frac3{32}}A^{-\frac98}\Gamma\left(\frac14\right)^{-\frac34}.
\end{equation}

In the subsequent, we will show that Corollary \ref{c(1/6)} in fact implies a relation between the special values of Barnes' $G$-function and the Dirichlet $L$-function.
Putting $x=\frac13$ in (\ref{log-C2-Cl}) and by simplifying, we get
\begin{equation}\label{log-C2-Cl-1/3}
\log\C_2\left(\frac16\right)
=\frac {2\log(2\pi)}3 -\frac{\log3}4
+\log\frac{G \left(\frac13\right)}{G \left(\frac53\right)}.
\end{equation}
Then by substituting the following identity (see  Corollary \ref{c(1/6)})
\begin{equation}\label{C2-Cl-1/6}
\C_2\left(\frac16\right)=3^{\frac1{12}}\exp\left(\frac{\sqrt3}{4\pi}\left(\frac14L(2,\chi_3)-L(2,\chi_6)\right)\right)
\end{equation}
into (\ref{log-C2-Cl-1/3}), we obtain
\begin{equation}\label{C2-Cl-1/6-2}
\log\frac{G \left(\frac13\right)}{G \left(\frac53\right)}
=\frac{\log3}3-\frac{2\log(2\pi)}3+\frac{\sqrt3}{4\pi}\left(\frac14L(2,\chi_3)-L(2,\chi_6)\right).
\end{equation}
If going on substituting the following formula (see \cite[p. 16]{Ad})
\begin{equation}\label{G-1/3}
\log G \left(\frac13\right)=\frac{\log3}{72}+\frac\pi{18\sqrt3}-\frac23\log\Gamma \left(\frac13\right)
-\frac{4\log A}3-\frac{\psi^{(1)} \left(\frac13\right)}{12\pi\sqrt{3}}+\frac19
\end{equation}
into (\ref{C2-Cl-1/6-2}),
we further get 
\begin{equation}\label{G-5/3}
\begin{aligned}
\log G \left(\frac53\right)&=-\frac{23\log3}{72}+\frac\pi{18\sqrt3}-\frac23\log\Gamma \left(\frac13\right)
-\frac{4\log A}3-\frac{\psi^{(1)} \left(\frac13\right)}{12\pi\sqrt{3}}  \\
&\quad+\frac19 +\frac{2\log(2\pi)}3
-\frac{\sqrt3}{4\pi}\left(\frac14L(2,\chi_3)-L(2,\chi_6)\right),
\end{aligned}
\end{equation}
where $\psi^{(1)}(x)=\frac{\partial\log\Gamma(x)}{\partial x^2}$ is the polygamma function.

From the definition of the triple cosine function (\ref{mcos-ex}) and the power series expansion
$$\log(1+x)=\sum_{k=1}^\infty(-1)^{k-1}\frac{x^k}{k}$$
for $|x|<1,$
we  may represent $\log\C_3(x)$ as
\begin{equation}\label{mcos-3}
\begin{aligned}
\log\C_3(x)&=\sum_{n=1}^\infty \left(\left(\frac{2n-1}{2}\right)^2 \log\left(1-\frac{4x^2}{(2n-1)^2}\right) +x^2\right) \\
&=\sum_{n=1}^\infty \left(-\frac{(2n-1)^2}{4}\sum_{k=1}^\infty\frac{4^k}{(2n-1)^{2k}}\frac{x^{2k}}{k} +x^2\right) \\
&=-\sum_{k=2}^\infty 4^{k-1}\lambda(2k-2)\frac{x^{2k}}{k},
\end{aligned}
\end{equation}
which is equivalent to
\begin{equation}\label{mcos-3-1}
\log\C_3(x)=-\sum_{k=1}^\infty 2^{2k}\lambda(2k)\frac{x^{2k+2}}{k+1}.
\end{equation}
Then by combining (\ref{lam}) and (\ref{mcos-3-1}), and using (531) in \cite[p. 221]{SC}, 
$\log\C_3(x)$  has the following representation
\begin{equation}\label{sc-com}
\begin{aligned}
\log\C_3(x)&=-\log\left(2^{-\frac1{24}}\cdot e^{-\frac18}\cdot A^{\frac32}\right) \\
&\quad-(1-\log(2\pi))\frac{x^2}2-\frac14\log\Gamma\left(\frac12+x\right)\Gamma\left(\frac12-x\right) \\
&\quad-\left(\frac12+x\right)\log G\left(\frac12+x\right)-\left(\frac12-x\right)\log G\left(\frac12-x\right) \\
&\quad+\int_0^x\log G\left(t+\frac12\right)dt+\int_0^{-x}\log G\left(t+\frac12\right)dt
\end{aligned}
\end{equation}
for $|x|<\frac12,$
where $A$ is the Glaisher-Kinkelin constant 
(also see \cite[p. 25]{SC}).

(\ref{mcos-3-1}) can be generalized from $r=3$ to $r=2,3,4,\ldots$ as follows.
Recall that (see Proposition \ref{pro1})
\begin{equation}\label{logc3}
\log\C_r(x)=-\int_0^x\pi t^{r-1}\tan(\pi t)dt
\end{equation}
and notice the following well-known identity
\begin{equation}\label{well}
t\tan (t)=\sum_{k=1}^\infty\frac{(-1)^{k-1}2^{2k}(2^{2k}-1)B_{2k}}{(2k)!}t^{2k}
\end{equation}
for $|t|<\frac{\pi}2.$ By combining (\ref{well}) with (\ref{zeta-even}), and setting $t\to\pi t,$ we have
\begin{equation}\label{tan-3}
\pi t\tan(\pi t)=2\sum_{k=1}^\infty2^{2k}\lambda(2k)t^{2k}
\end{equation}
for $|t|<\frac12,$ where we have used the relation $\lambda(s)=(1-2^{-s})\zeta(s)$ (see (\ref{lam})).
Then for $r=2,3,4,\ldots,$ by multiplying (\ref{tan-3}) with $t^{r-2}$ and integrating the result equation, from (\ref{logc3}) we get
\begin{equation}\label{mcos-3-r}
\log\C_r(x)=-2\sum_{k=1}^\infty 2^{2k}\lambda(2k)\frac{x^{2k+r-1}}{2k+r-1}.
\end{equation}
And this representation is valid for $|x|<\frac12.$

Setting $x=\frac14$ in (\ref{mcos-3-r}), and using Corollary \ref{rem-ex} for $r=2,3,4$ and 5, it is readily to obtain 
$$
\begin{aligned}
&\sum_{k=1}^\infty\frac{\lambda(2k)}{(2k+1)2^{2k}}=-\frac{\log2}{4}+\frac{G}{\pi}, \quad\text{(see \cite[p. 241, (666)]{SC})} 
\\
&\sum_{k=1}^\infty\frac{\lambda(2k)}{(2k+2)2^{2k}}=-\frac{\log2}{4}+\frac{2G}{\pi}-\frac{7\zeta_E(3)}{2\pi^2},  
\\
&\sum_{k=1}^\infty\frac{\lambda(2k)}{(2k+3)2^{2k}}=-\frac{\log2}{4} +\frac{3G}{\pi}+\frac{3\zeta_E(3)}{2\pi^2}
-\frac{24\beta(4)}{\pi^3}, 
\\
&\sum_{k=1}^\infty\frac{\lambda(2k)}{(2k+4)2^{2k}}=-\frac{\log2}{4} +\frac{4G}{\pi}+\frac{3\zeta_E(3)}{\pi^2}
-\frac{96\beta(4)}{\pi^3}+\frac{186\zeta_E(5)}{\pi^4}. \label{l-r=5}
\end{aligned}
$$
Furthermore, by subtracting the above series we get
$$
\begin{aligned}
&\sum_{k=1}^\infty\frac{\lambda(2k)}{(2k+1)(2k+2)2^{2k}}=-\frac{G}{\pi}+ \frac{7\zeta_E(3)}{2\pi^2}, 
\\
&\sum_{k=1}^\infty\frac{\lambda(2k)}{(2k+2)(2k+3)2^{2k}}=-\frac{G}{\pi}-\frac{5\zeta_E(3)}{\pi^2}+\frac{24\beta(4)}{\pi^3}, 
\\
&\sum_{k=1}^\infty\frac{\lambda(2k)}{(2k+3)(2k+4)2^{2k}}=-\frac{G}{\pi}-\frac{3\zeta_E(3)}{2\pi^2}
+\frac{72\beta(4)}{\pi^3}-\frac{186\zeta_E(5)}{\pi^4}. \label{l-r=4}
\end{aligned}
$$
From this, we have the following series representation for $\zeta_E(3)$:
$$\zeta_E(3)=\frac{2\pi^2}{7}\left(\frac G\pi+\sum_{k=1}^\infty\frac{\lambda(2k)}{(2k+1)(2k+2)2^{2k}}\right).$$

\section*{Acknowledgement}  We are grateful to Professor Jean-Paul Allouche for his interested in this work and for his many helpful
comments and suggestions.

\end{document}